\documentclass[12 pt]{amsart}
\usepackage{amsmath,times,epsfig,amssymb,amsbsy,amscd,amsfonts,amstext,graphicx,color,bm}
\usepackage[arrow,matrix]{xy}
\usepackage{graphicx}

\newcommand{\C}{\mathbb{C}}

\newcommand{\R}{\mathbb{R}}

\newcommand{\Z}{\mathbb{Z}}

\newtheorem{Theorem}{Theorem}[section]
\newtheorem{Definition}[Theorem]{Definition}
\newtheorem{Assumption}[Theorem]{Assumption}
\newtheorem{Lemma}[Theorem]{Lemma}
\newtheorem{Proposition}[Theorem]{Proposition}

\newtheorem{Remark}[Theorem]{Remark}
\newtheorem{Example}[Theorem]{Example}

\DeclareMathOperator{\image}{Im}

\DeclareMathOperator{\codim}{codim}
\DeclareMathOperator{\Hom}{Hom}

\newcommand{\A}{\mathcal{A}}

\newcommand{\ch}{\mathsf{ch}}

\newcommand{\bC}{\ensuremath{\mathbb{C}}}

\newcommand{\bF}{\ensuremath{\mathbb{F}}}

\newcommand{\bK}{\ensuremath{\mathbb{K}}}

\newcommand{\bP}{\ensuremath{\mathbb{P}}}
\newcommand{\bQ}{\ensuremath{\mathbb{Q}}}
\newcommand{\bR}{\ensuremath{\mathbb{R}}}

\newcommand{\bT}{\ensuremath{\mathbb{T}}}

\newcommand{\bZ}{\ensuremath{\mathbb{Z}}}
\newcommand{\scA}{\ensuremath{\mathcal{A}}}

\newcommand{\scC}{\ensuremath{\mathcal{C}}}

\newcommand{\scF}{\ensuremath{\mathcal{F}}}

\newcommand{\scL}{\ensuremath{\mathcal{L}}}

\newcommand{\scS}{\ensuremath{\mathcal{S}}}

\newcommand{\scV}{\ensuremath{\mathcal{V}}}

\newcommand{\Sep}{\operatorname{Sep}}
\newcommand{\Ker}{\operatorname{Ker}}

\newcommand{\RB}{\operatorname{RB}}
\DeclareMathOperator{\rank}{rank}

\newcommand{\dec}{\mathsf{d}}

\begin{document}

\title[4-nets]{Resonant bands, Aomoto complex, and real $4$-nets}

\begin{abstract}
The resonant band is a useful notion for the computation of 
the nontrivial monodromy eigenspaces of the Milnor fiber of 
a real line arrangement. In this article, 
we develop the resonant band description for the cohomology 
of the Aomoto complex. As an application, we prove that real 
$4$-nets do not exist. 
\end{abstract}

\author{Michele Torielli}
\address{Michele Torielli, Department of Mathematics, Hokkaido University, Kita 10, Nishi 8, Kita-Ku, Sapporo 060-0810, Japan.}
\email{torielli@math.sci.hokudai.ac.jp}

\author{Masahiko Yoshinaga}
\address{Masahiko Yoshinaga, Department of Mathematics, Hokkaido University, Kita 10, Nishi 8, Kita-Ku, Sapporo 060-0810, Japan.}
\email{yoshinaga@math.sci.hokudai.ac.jp}


\date{\today}
\maketitle

\tableofcontents

\section{Introduction}

Combinatorial decisions of topological invariants are 
the central problems in the theory of hyperplane arrangements. 
Milnor fibers and their eigenspace decompositions have 
received a lot of attention and have been studied 
by diverse techniques (\cite{suc-mil}) (e.g., 
Alexander polynomials, 
Hodge theory, 
nets and multinets, covering spaces, Salvetti complexes, 
characteristic and resonance varieties etc.) 
Among others, the authors follow the previous studies 
using real structures, (\cite{williams,yos-mil,yos-res}) and 
Aomoto complex over finite fields, 
\cite{coh-orl, gra,mac-pap, ps-sp}. 

Concerning the relation between Milnor fibers and 
Aomoto complexes, the two key results were obtained 
by Papadima and Suciu \cite{ps-sp, ps-mod}. 
\[
\mbox{Monodromy eigenspaces}
\stackrel{\mbox{(1)}}{\longleftrightarrow}
\mbox{Aomoto complex}
\stackrel{\mbox{(2)}}{\longleftrightarrow}
\mbox{Multinets}
\]
The first one is an upper bound for the rank of eigenspace 
in terms of the Betti numbers of the Aomoto complexes 
over finite fields \cite{ps-sp}. It was subsequently used 
by many authors to prove vanishing theorems 
\cite{bailet, bai-yos, dim, mac-pap}. 
The second one is the bijective correspondence between 
$3$-nets and nonzero elements in the cohomology group 
of the Aomoto complex over $\bF_3$. 
A degree one element of the Orlik-Solomon algebra over 
the finite field $\bF_q$ is bijectively corresponding to 
the coloring (with $q$-colors) of the arrangement. 
Papadima and Suciu succeeded to translate the cocycle condition 
into combinatorics of coloring \cite{ps-mod}. 
The deep relation between nontrivial eigenspaces and multinet 
structure had been conjectured. Papadima-Suciu's results 
provide a beautiful framework to understand the nontrivial 
eigenspaces via multinets. 

If we restrict our attention to real arrangements, the real 
structure contains a lot of information about 
topology of the complexification. The resonant band, 
introduced in \cite{yos-mil, yos-res}, is a useful tool 
for computing nontrivial eigenspace of the Milnor fibers and 
local system cohomology groups. 
The purpose of this paper is to introduce the notion of 
resonant bands for the Aomoto complex (over any coefficient ring) 
of a real arrangements. Then combining resonant bands techniques with 
the above Papadima-Suciu's picture (over $\bF_2$), we prove 
that real $4$-nets do not exist, which is a 
partial answer to a conjecture that the Hessian arrangement is 
the only $4$-net. 

The paper is organized as follows. \S\ref{sec:pre} is a summary 
of well known facts on multinets and Orlik-Solomon algebras. 
Especially, we describe in detail the transformation of the Orlik-Solomon 
algebra when we exchange the hyperplane at infinity, which will be 
used later. \S\ref{sec:pscorresp} is a summary of the recent work by 
Papadima-Suciu. The crucial result that we use later is 
Theorem \ref{thm:mod2ps}. 
Theorem \ref{thm:mod2ps} translates the cocycle conditions 
of the Aomoto complex (over $\bF_2$) into combinatorial structures of 
subarrangements. \S\ref{sec:resband} is the main part of this paper. 
After recalling a description of the Aomoto complex in terms of 
chambers in \S\ref{subsec:viacha} (following 
\cite{yos-cha}), we introduce the notion of 
$\eta$-resonant band in \S\ref{subsec:viaresban}. In the 
main theorem (Theorem \ref{thm:resbanAomoto}), we prove that 
the cohomology of the Aomoto complex is isomorphic to a submodule 
of the free module generated by resonant bands under certain 
non-resonant condition at infinity. When the coefficient 
ring of the Aomoto complex is $\bF_2$, 
everything can be 
described in terms of combinatorics of subarrangements. 
This translation is done in \S\ref{subsec:subarr}. 
In \S\ref{sec:4net}, we prove the non-existence of real $4$-nets. 
The key result is the Non-Separation Theorem \ref{thm:nonsep} 
in \S\ref{subsec:diag} which concerns subarrangements corresponding 
to the cocycle of the Aomoto complex over $\bF_2$. The Non-Separation 
Theorem asserts that at the intersection of multiplicity $4$, the 
subarrangement corresponding to a nontrivial cohomology class 
has special ordering. This assertion heavily relies on the 
real structure. Therefore, at this moment, it seems hopeless 
to generalize our argument to the complex case. 
If there exists a real $4$-net, it is easy to 
construct a subarrangement 
which contradicts the Non-Separation Theorem. Hence 
real $4$-nets do not exist (\S\ref{subsec:4net}). 
(This fact was first proved by Cordovil-Forge \cite[Lem. 2.4]{cor-for}. 
Our arguments prove a little bit stronger version. 
See Remark \ref{rem:pseudo}.)

\section{Preliminaries}
\label{sec:pre}

\subsection{Conventions}
\label{sec:notation}
In this paper, three types of hyperplane arrangements appear: 
affine arrangements in $\bK^\ell$, hyperplane arrangements in 
the projective space $\bK\bP^\ell$ and central arrangements 
in $\bK^{\ell+1}$. 
It is better to distinguish by notations (\cite{ot-arr, ot-int}). 

\begin{itemize}
\item 
$\A=\{H_1, \dots, H_n\}$ denotes an arrangement of 
affine hyperplanes in the affine $\ell$-space $\bK^\ell$. 
\item 
$\widetilde{\A}=c\A=\{\widetilde{H}_0, \widetilde{H}_1, 
\dots, \widetilde{H}_n\}$ denotes the coning of 
$\A$, which is a central hyperplane 
arrangement in $\bK^{\ell+1}$. 
The hyperplane $\widetilde{H}_0$ is corresponding to the 
hyperplane at infinity of $\A$. 
\item 
$\overline{\A}=\{\overline{H}_0, \overline{H}_1, 
\dots, \overline{H}_n\}$ denotes the projectivization 
of $\widetilde{\A}$, 
which is a hyperplane arrangement 
on the projective $\ell$-space $\bK\bP^\ell$ induced by 
$\widetilde{\A}$. 
\item 
$\dec_{\widetilde{H_i}}\widetilde{\A}=\{\dec_{\widetilde{H}_i}\widetilde{H}_0, \dots, \widehat{\dec_{\widetilde{H}_i}\widetilde{H}_i}, 
\dots, \dec_{\widetilde{H}_i}\widetilde{H}_n\}$ denotes the deconing 
of $\widetilde{\A}$ with respect to the hyperplane $\widetilde{H}_i$. 
Note that $\dec_{\widetilde{H}_0}\widetilde{\A}=\A$. 
\end{itemize}
Other frequently used notations are: 
\begin{itemize}
\item 
$R$: a commutative ring (unless stated otherwise), 
\item 
$\bK$: a field, 
\item 
$M(\A)$: the complexified complement of $\A$. 
\end{itemize}

\subsection{Multinets}
\label{sec:milnor}

In this subsection, we recall several facts on multinets. 

\begin{Definition}
Let $\overline{\A}=\{\overline{H}_0, \dots, \overline{H}_n\}$ be a projective line arrangement in $\C\bP^2$. 
Let $k\geq 3$ and $d\geq 2$ be integers. 
A {\em (reduced) $(k,d)$-multinet} (or {\em $k$-multinet} for simplicity) 
on $\A$ is a pair 
$(\mathcal{N}, \mathcal{X})$, where $\mathcal{N}$ is a 
partition of $\A$ into $k$ classes 
$\overline{\A}=
\overline{\A}_1\sqcup\dots\sqcup \overline{\A}_k$ 
and $\mathcal{X}\subset\C\bP^2$ 
is a set of multiple points (called the base locus) 
such that 
\begin{itemize}
\item[$(i)$] $|\overline{\A}_1|=\dots=|\overline{\A}_k|=:d$; 
\item[$(ii)$] $\overline{H}\in\overline{\A}_i$ and $\overline{H}'\in\overline{\A}_j$ ($i\neq j$) imply that 
$\overline{H}\cap \overline{H}'\in\mathcal{X}$; 
\item[$(iii)$] for all $p\in\mathcal{X}$, $n_p:=|\{\overline{H}\in\overline{\A}_i\mid \overline{H}\ni p\}|$ 
is constant and independent of $i$;  
\item[$(iv)$] for any $\overline{H}, \overline{H}'\in\overline{\A}_i$ ($i=1, \dots, k$), there is a sequence $\overline{H}=\overline{H}'_0, \overline{H}'_1, \dots, \overline{H}'_r=\overline{H}'$ in $\overline{\A}_i$ such that $\overline{H}'_{j-1}\cap \overline{H}'_j\notin\mathcal{X}$ for $1\leq j\leq r$. 
\end{itemize}
If $n_p = 1$ for every $p\in\mathcal{X}$, then $(\mathcal{N}, \mathcal{X})$ is called a \emph{net}. 
\end{Definition}
Note that if 
$(\mathcal{N}, \mathcal{X})$ is a $(k,d)$-net, 
then each $p\in\mathcal{X}$ has multiplicity $k$. 

The next theorem, which combines results of Pereira and Yuzvinsky 
\cite{per-yuz, yuz-new}, 
summarizes what is known about the existence 
of non-trivial multinets on arrangements 
(see also \cite{suc-mil,bar-yuz,yuz-ab} for more results). 
\begin{Theorem}Let $\overline{\A}$ be a $k$-multinet, with base locus $\mathcal{X}$. Then
\begin{enumerate}
\item If $|\mathcal{X}|>1$, then $k=3$ or $4$. 
\item If there is a hyperplane $\overline{H}\in \overline{\A}$ such that $m_H >1$, then $k=3$. 
\item If $k=4$, then $|\mathcal{X}|=d^2$ and it is a $(4,d)$-net. 
\end{enumerate}
\end{Theorem}
Although several infinite families of multinets with $k = 3$ are known, 
only one multinet with $k = 4$ is known to exist: 
the $(4, 3)$-net on the Hessian arrangement (which is defined 
over $\bQ(\sqrt{-3})$). 
It is conjectured that 
the only $(4,d)$-net is the Hessian arrangement. 
In \cite{dmwz}, it is proved that the Hessian is the 
unique $(4,d)$-net for $d\leq 6$ (hence for $|\A|\leq 24$). 
We will later prove that there does not exist real $(4,d)$-net 
for any $d$.

\subsection{Orlik-Solomon algebra and Aomoto complex}

Let $\A=\{H_1, \dots, H_n\}$ be an arrangement of 
affine hyperplanes in $\bK^{\ell}$ 
and $R$ be a commutative ring. 
Let $E_1=\bigoplus_{j=1}^nRe_j$ be the 
free module generated by $e_1, e_2, \dots, e_n$, where $e_i$ is 
a symbol corresponding to the hyperplane $H_i$. 
Let $E=\wedge E_1$ be the exterior algebra over $R$. 
The algebra $E$ is graded via $E=\bigoplus_{p=0}^nE_p$, 
where $E_p=\wedge^pE_1$. 
The $R$-module $E_p$ is free and has the 
distinguished basis consisting of monomials 
$e_S:=e_{i_1}\wedge\cdots\wedge e_{i_p}$, 
where $S=\{{i_1},\dots, {i_p}\}$ is 
running through all the subsets of 
$\{1,\dots,n\}$ of cardinality $p$ and $i_1<i_2<\cdots<i_p$. 
The graded algebra $E$ is a commutative DGA with respect to 
the differential $\partial$ of degree $-1$ uniquely defined 
by the conditions $\partial e_i=1$ for all $i=1,\dots, n$ and 
the graded Leibniz formula. Then for every $S\subset\{1,\dots,n\}$ 
of cardinality $p$
\[
\partial e_S=\sum_{j=1}^p(-1)^{j-1}e_{S_j},
\]
where $S_j$ is the complement in $S$ to its $j$-th element. 

For every $S\subset\{1,\dots,n\}$, put 
$\cap S=\bigcap_{i\in S}H_i$ (possibly $\cap S=\emptyset$). The set of all 
intersections 
$L(\A)=\{\cap S\mid S\subset \{1,\dots,n\}\}$ 
is called 
the intersection poset. 
The subset $S\subset\{1,\dots,n\}$ is called 
\emph{dependent} if $\cap S\ne\emptyset$ 
and the set of linear polynomials $\{\alpha_i~|~i\in S\}$ with 
$H_i=\alpha_i^{-1}(0)$, is linearly dependent.
\begin{Definition}\label{def:osalgbr} 
The \emph{Orlik-Solomon ideal} of $\A$ is the ideal $I=I(\A)$ of 
$E$ generated by 
\begin{itemize}
\item[$(1)$] 
all $e_S$ with $\cap S=\emptyset$ and 
\item[$(2)$] 
all $\partial e_S$ with $S$ dependent. 
\end{itemize}
The algebra $A=A_R^\bullet(\A)=E/I(\A)$ is called the 
\emph{Orlik-Solomon algebra} of $\A$.
\end{Definition}
Clearly $I$ is a homogeneous ideal of $E$ whence $A$ is a graded algebra and we can write $A=\bigoplus A_R^p$, where $A_R^p=E_p/(I\cap E_p)$.
If $\A$ is central, then for any $S\subset\A$, we have 
$\cap S\neq\emptyset$. Therefore, the Orlik-Solomon ideal is generated 
by the element of type $(2)$ of Definition \ref{def:osalgbr}. 
In this case, the map 
$\partial$ induces a well-defined differential 
$\partial\colon A_R^\bullet(\A)\longrightarrow A_R^{\bullet -1}(\A)$. 

Notice that, for each $p$, we can write (Brieskorn decomposition) 
\begin{equation}
\label{eq:brieskorn}
A_R^p(\A)=\bigoplus_{X\in L_p(\A)}A_R^p(\A_X),
\end{equation}
where $L_p(\A):=\{X\in L(\A)~|~\codim X=p\}$ and 
$\A_X:=\{H\in\A~|~X\subset H\}$. 

Recall that the coning $\widetilde{\A}=c\A=\{\widetilde{H}_0, 
\widetilde{H}_1, \dots, \widetilde{H}_n\}$ of $\A$ is a central 
arrangement in $\bK^{\ell+1}$. We denote the corresponding 
generators of Orlik-Solomon algebra $A_R^\bullet(\widetilde{\A})$ by 
$\widetilde{e}_0, \widetilde{e}_1, \dots, \widetilde{e}_n$. 
The map 
\[
\iota\colon
A_R^1(\A)\longrightarrow A_R^1(\widetilde{\A}):
e_i\longmapsto \widetilde{e}_i-\widetilde{e}_0,
\]
induces an injective $R$-algebra homomorphism 
$\iota\colon A_R^\bullet(\A)\longrightarrow A_R^\bullet(\widetilde{\A})$ 
(\cite{yuz-surv}). 
The image of the embedding $\iota$ is equal to the subalgebra 
\[
A_R^\bullet(\widetilde{\A})_0:=
\{\omega\in A_R^\bullet(\widetilde{\A})\mid \partial(\omega)=0\} 
\]
of $A_R^\bullet(\widetilde{\A})$. 
Consider the deconing 
$\A':=\dec_{\widetilde{H}_i}\widetilde{\A}=
\{H_0', \dots, \widehat{H_i'}, \dots, H_n'\}$ 
with respect to the hyperplane 
$\widetilde{H}_i\in \widetilde{\A}$. We denote the generators of 
Orlik-Solomon algebra $A_R^\bullet(\A')$ by $e_0', \dots, 
\widehat{e_i'}, \dots, e_n'$. 
Then the Orlik-Solomon algebras of $\A$ and $\A'$ are 
isomorphic $A_R^\bullet(\A)\simeq A_R^\bullet(\A')$. The 
explicit isomorphism is given by 
\[
e_j\longmapsto
\left\{
\begin{array}{cl}
e_j'-e_0',&\mbox{ if }1\leq j\leq n, j\neq i, \\
-e_0',&\mbox{ if }j=i. 
\end{array}
\right.
\]

Let us fix an element $\eta=\sum_{i=1}^na_ie_e
\in A_R^1(\A)$. Since $\eta\wedge\eta=0$, 
\[
0\longrightarrow
A_R^1(\A)\stackrel{\eta}{\longrightarrow}
A_R^2(\A)\stackrel{\eta}{\longrightarrow}\dots
\stackrel{\eta}{\longrightarrow}
A_R^\ell(\A)\stackrel{\eta}{\longrightarrow}0
\]
forms a cochain complex, which is called the 
\emph{Aomoto complex} $(A_R^\bullet(\A), \eta)$. 
By the above embedding $\iota$, 
we can identify the Aomoto complex 
$(A_R^\bullet(\A), \eta)$ with 
$(A_R^\bullet(\widetilde{\A})_0, \widetilde{\eta})$, where 
$\widetilde{\eta}=\iota(\eta)=\sum_{i=1}^na_i\widetilde{e}_i-
(a_1+\dots+a_n)\widetilde{e}_0$, (\cite{fal-coh}).

\section{Mod $p$ Aomoto complex and Papadima-Suciu correspondence}
\label{sec:pscorresp}

In this section, we recall a recent work by Papadima and Suciu 
\cite{ps-mod}. They found a way of constructing $3$-net from 
an element of the cohomology of Aomoto complex over $\bF_3$. 
Let $\overline{\A}=\{\overline{H}_0, \dots, \overline{H}_n\}$ 
be a line arrangement on the projective plane 
$\bK\bP^2$ with $3|\sharp(\overline{\A})$. Assume that 
there do not exist multiple points of multiplicity 
$\{3r\mid r\in\bZ, r>1\}$. Let 
$\widetilde{\eta}_0:=\sum_{i=0}^n\widetilde{e}_i\in
A_{\bF_3}^1(\widetilde{\A})_0$ be the diagonal element. Then 
there is a natural bijective correspondence: 
\begin{equation}
\label{eq:pscorresp}
(H^1(A_{\bF_3}^\bullet(\widetilde{\A})_0, \widetilde{\eta}_0)\setminus\{0\})
/\bF_3^\times\stackrel{\simeq}{\longrightarrow}
\left\{
\begin{array}{c}
\mbox{Isomorphism classes of}\\
\mbox{$3$-net structures on $\overline{\A}$} 
\end{array}
\right\}. 
\end{equation}
The correspondence is explicitly given by 
$H^1(A_{\bF_3}^\bullet(\widetilde{\A})_0, \widetilde{\eta}_0)\ni 
\omega=\sum_{k=0}^n a_i\widetilde{e}_i \longmapsto
(\overline{\A}_0,\overline{\A}_1,\overline{\A}_2)$, 
where $\overline{\A}_m=\{\overline{H}_i\mid a_i=m\}$ ($m=0,1,2$). 
The point of the above correspondence is that by using 
the local structures of the Orlik-Solomon algebra, we can translate 
the cocycle condition into the combinatorial condition of 
$(\overline{\A}_0,\overline{\A}_1,\overline{\A}_2)$, which 
turns out to be exactly the defining conditions of $3$-nets. 
Later we will employ a similar consideration for the 
Aomoto complex over $\bF_2$ which we summarize in 
this section.

\subsection{A local lemma}
\label{subsec:loclem}

To analyze the map $\eta\colon A_R^1(\A)\longrightarrow A_R^2(\A)$ by 
the Brieskorn decomposition (\ref{eq:brieskorn}), the next 
lemma is useful (cf. \cite[\S3]{lib-yuz}). 

\begin{Lemma}
\label{lem:local}
Let $\scC_s=\{H_1, \dots, H_s\}$ be a central arrangement in 
$\bK^2$ (Figure \ref{fig:central}). 
Let $R$ be a commutative ring and 
$\eta=a_1e_1+\dots +a_se_s\in A_R^1(\scC_s)$ be a degree one 
element of Orlik-Solomon algebra. 
\begin{itemize}
\item[(1)] 
$\eta\wedge(e_i-e_j)=-(\sum_{i=1}^s a_i)\cdot e_i\wedge e_j$. 
\item[(2)] 
Let $\omega=b_1e_1+\dots +b_se_s\in A_R^1(\A)$ be another element. 
Assume that $\omega$ and $\eta$ are linearly independent 
(i.e.,  $c_1\eta+c_2\omega=0, (c_1, c_2\in R) \Longrightarrow 
c_1=c_2=0$). Then $\eta\wedge\omega=0$ if and only if 
$\sum_{i=1}^s a_i=\sum_{i=1}^s b_i=0$. 
\end{itemize}

\begin{figure}[htbp]
\begin{picture}(400,100)(0,0)
\thicklines

\put(140,10){\line(1,1){90.1}}
\put(160,10){\line(2,3){60.1}}
\put(180,10){\line(1,3){30.1}}

\multiput(195,4)(10,0){4}{\circle*{3}}

\put(240,10){\line(-2,3){61}}

\put(132,-0.1){$H_1$}
\put(152,-0.1){$H_2$}
\put(172,-0.1){$H_3$}
\put(240,-0.1){$H_s$}

\end{picture}
      \caption{Central arrangement $\scC_s$}
\label{fig:central}
\end{figure}

\end{Lemma}

\begin{proof}
(1) It is straightforward from the relation $e_{ij}=e_{ik}-e_{jk}$, where 
$e_{ij}:=e_i\wedge e_j$. 

(2) If $\sum_{i=1}^s b_i=0$, then 
$\omega=b_1(e_1-e_s)+\dots+b_{s-1}(e_1-e_{s-1})$. 
Then applying (1), we have $\eta\wedge\omega=
-(\sum_{i=1}^sa_i)\cdot(b_1e_{1s}+\dots +b_{s-1}e_{1,s-1})$. 
This is zero if $\sum_{i=1}^s a_i=0$. 
Conversely, suppose $\eta\wedge\omega=0$. 
Since $\scC_s$ is central, we can apply the 
differential 
$\partial$. We have 
\[
0=\partial(\eta\wedge\omega)=(\partial\eta)\omega-
(\partial\omega)\eta. 
\]
By the assumption that $\eta$ and $\omega$ are linearly independent, 
$\partial\eta=\partial\omega=0$. 
\end{proof}

\subsection{Aomoto complex over $\bF_p$}
\label{subsec:pscorresp}

Let $\A=\{H_1, \dots, H_n\}$ be an arrangement of affine 
lines in $\bK^2$. Choose a prime $p$ such that 
$p|(n+1)$. Consider the Aomoto complex over 
$R=\bF_p$ and the embedding $\iota\colon A_{\bF_p}^\bullet(\A)
\stackrel{\simeq}{\longrightarrow} 
A_{\bF_p}^\bullet(\widetilde{\A})_0\subset A_{\bF_p}^\bullet(\widetilde{\A})$. 
Since $n$ is equal to $-1$ in $\bF_p$, the image of the 
diagonal element 
$\eta_0:=e_1+\dots+e_n\in A_{\bF_p}^1(\A)$ is 
\[
\widetilde{\eta}_0:=\iota(\eta_0)=
\widetilde{e}_0+\widetilde{e}_1+\dots+\widetilde{e}_n
\in A_{\bF_p}^1(\widetilde{\A})_0. 
\]
We consider the first cohomology group of the Aomoto complex 
$(A_{\bF_p}^\bullet(\A), \eta_0)\simeq
(A_{\bF_p}^\bullet(\widetilde{\A})_0, \widetilde{\eta}_0)$. Let 
$\widetilde{\omega}=
\sum_{i=0}^na_i\widetilde{e}_i\in A_{\bF_p}^1(\widetilde{\A})_0$. 
Let us translate the relation 
$\widetilde{\eta}\wedge\widetilde{\omega}=0$ in terms 
of coefficients $a_i$ of $\widetilde{\omega}$ 
by using the Brieskorn decomposition 
(\ref{eq:brieskorn}). 
For an intersection $X\in L_2(\widetilde{\A})$ of 
codimension two, let us define the localization at $X$ by 
\begin{equation}
\label{eq:proj}
\widetilde{\omega}|_X:=
\sum_{\widetilde{H}_i\in\widetilde{\A}_X}a_i\widetilde{e}_i. 
\end{equation}

\begin{Proposition}
\label{prop:zeroprod}
With notation as above, $\widetilde{\eta}_0\wedge
\widetilde{\omega}=0$ 
if and only if the following $(i)$ and $(ii)$ hold. 
\begin{itemize}
\item[$(i)$] 
Let $X\in L_2(\widetilde{\A})$. If $|\widetilde{\A}_X|$ is 
divisible by $p$, then $
\sum_{\widetilde{H}_i\in\widetilde{\A}_X}a_i=0$ in $\bF_p$. 
\item[$(ii)$] 
Let $X\in L_2(\widetilde{\A})$. If $|\widetilde{\A}_X|$ is not 
divisible by $p$, then 
\[
a_{i_1}=a_{i_2}=\dots=a_{i_t}, 
\]
where $\widetilde{\A}_X=\{\widetilde{H}_{i_1}, \widetilde{H}_{i_2}, \dots, 
\widetilde{H}_{i_t}\}$. (This is equivalent to that 
$\widetilde{\omega}|_X$ and $\widetilde{\eta}_0|_X$ are linearly dependent.) 
\end{itemize}
\end{Proposition}

\begin{proof}
By the Brieskorn decomposition (\ref{eq:brieskorn}), 
$\widetilde{\eta}_0\wedge\widetilde{\omega}=0$ if and only if 
$\widetilde{\eta}_0|_X\wedge\omega|_X=0$ for all $X\in L_2(\widetilde{\A})$. 
Using the Lemma \ref{lem:local} (2), it is equivalent to 
$(i)$ and $(ii)$ above. 
\end{proof}

\subsection{Aomoto complex over $\bF_2$ and subarrangements}
\label{subsec:mod2}

Now we consider the Aomoto complex over $\bF_2=\bZ/2\bZ$. 
Since the coefficient is either $0$ or $1\in\bF_2$, 
elements of $A_{\bF_2}^1(\widetilde{\A})$ can be identified with 
subarrangements of $\widetilde{\A}$. 

\begin{Definition}
Let $\widetilde{\scS}\subset\widetilde{\A}$ be a subset. 
Let us define an element corresponding to the subset by 
\[
\widetilde{e}(\widetilde{\scS}):=
\sum_{\widetilde{H}_i\in\widetilde{\scS}}\widetilde{e}_i
\in A_{\bF_2}^1(\widetilde{\A}). 
\]
For an affine arrangement $\A=\{H_1, \dots, H_n\}$ and a 
subset $\scS\subset\A$, similarly we define 
$e(\scS):=\sum_{H_i\in\scS}e_i\in A_{\bF_2}^1(\A)$. 
\end{Definition}
Obviously the diagonal element is 
$\widetilde{\eta}_0=\widetilde{e}(\widetilde{\A})$ and 
$\widetilde{e}(\widetilde{\scS})+\widetilde{\eta}_0=
\widetilde{e}(\widetilde{\A}\setminus\widetilde{\scS})$. 

Applying Proposition \ref{prop:zeroprod} for $p=2$, we have the following. 

\begin{Theorem}
\label{thm:mod2ps}
Let $\widetilde{\A}=\{\widetilde{H}_0, \widetilde{H}_1, \dots, 
\widetilde{H}_n\}$ be central arrangement in $\bK^3$. 
Let $\widetilde{\scS}\subset\widetilde{\A}$ be a subset. 
Then 
$\widetilde{\eta}_0\wedge
\widetilde{e}(\widetilde{\scS})=0$ if and only if 
the following $(i)$ and $(ii)$ hold. 
\begin{itemize}
\item[$(i)$] 
Let $X\in L_2(\widetilde{\A})$. 
If $|\widetilde{\A}_X|$ is even, then 
$|\widetilde{\scS}_X|$ is also even. 
\item[$(ii)$] 
Let $X\in L_2(\widetilde{\A})$. 
If $|\widetilde{\A}_X|$ is odd, then either 
$\widetilde{\A}_X=\widetilde{\scS}_X$ or 
$\widetilde{\scS}_X=\emptyset$. 
\end{itemize}
\end{Theorem}

\begin{Remark}
\label{rem:equivetapart} 
The existence of 
$\widetilde{\omega}\in A_{\bF_2}^1(\widetilde{\A})$ 
such that $\widetilde{\omega}\neq 0, 
\widetilde{\omega}\neq\widetilde{\eta}_0$ and 
$\widetilde{\eta}_0\wedge\widetilde{\omega}=0$ is equivalent to the existence of a partition $\widetilde{\A}=\widetilde{\A}_1\sqcup\widetilde{\A}_2$ such that at each intersection $X\in L_2(\widetilde{\A})$ of codimension $2$, (at least) one of the following is satisfied: 
\begin{enumerate}
\item $\widetilde{\A}_X$ is included in $\widetilde{\A}_1$ or in $\widetilde{\A}_2$, 
\item $|(\widetilde{\A}_1)_X|$ and $|(\widetilde{\A}_2)_X|$ are both even. 
\end{enumerate}
The authors do not know any real essential arrangement which possesses the 
above partition. Hence, we do not know any real essential arrangement which 
satisfies $H^1(A_{\bF_2}^\bullet(\widetilde{\A})_0, \widetilde{\eta}_0)
\neq 0$. 
\end{Remark}

\begin{Example}
\label{ex:partition}
Suppose that $\overline{\A}=\overline{\A}_1\sqcup
\overline{\A}_2\sqcup\overline{\A}_3\sqcup\overline{\A}_4$ is 
a $4$-net. Then 
\[
\begin{split}
\widetilde{\eta}_0\wedge
\widetilde{e}
(\widetilde{\A}_1\cup\widetilde{\A}_2)&=
\widetilde{\eta}_0\wedge
\widetilde{e}
(\widetilde{\A}_1\cup\widetilde{\A}_3)\\
&=
\widetilde{\eta}_0\wedge
\widetilde{e}
(\widetilde{\A}_1\cup\widetilde{\A}_4)\\
&=0
\end{split}
\]
These three elements satisfy a linear relation, 
\[
\widetilde{e}
(\widetilde{\A}_1\cup\widetilde{\A}_2)+
\widetilde{e}
(\widetilde{\A}_1\cup\widetilde{\A}_3)+
\widetilde{e}
(\widetilde{\A}_1\cup\widetilde{\A}_4)
=\widetilde{\eta}_0, 
\]
and span a two dimensional subspace 
in $H^1(A_{\bF_2}^\bullet(\widetilde{\A})_0, \widetilde{\eta}_0)$. 
We obtain a well-known inequality 
$\dim H^1(A_{\bF_2}^\bullet(\widetilde{\A})_0, \widetilde{\eta}_0)\geq 2$ 
(\cite{dim-pap, ps-sp}). 
\end{Example}

\section{Resonant bands description of Aomoto complex}
\label{sec:resband}

Resonant bands provide effective tools to compute 
local system cohomology groups and eigenspaces of 
Milnor monodromies. In this section, 
we give a description of the cohomology of the Aomoto 
complex in terms of resonant bands. 

\subsection{Aomoto complex via chambers}
\label{subsec:viacha}

We first introduce several notions related to the 
real structure of line arrangements. 
(The notions are summarized in Example \ref{ex:chambers} and 
Figure \ref{fig:example}.) 
Let $\A=\{H_1, \dots, H_n\}$ be an arrangement of affine lines in 
$\bR^2$. 
A connected component of $\bR^2\setminus\bigcup_{H\in\scA}H$ 
is called a chamber. The set of all chambers is denoted by 
$\ch(\scA)$. Let $C, C'\in\ch(\A)$. The line $H\in\A$ 
is said to separate $C$ and $C'$ when they are contained in 
opposite sides of $H$. The set of 
all lines separating $C$ and $C'$ is denoted by $\Sep(C, C')$. The set of 
chambers $\ch(\A)$ is provided with a natural metric, the so-called 
adjacency distance, $d(C, C')=|\Sep(C, C')|$. 

Let us fix a flag 
\[
\emptyset=\scF^{-1}\subset
\scF^0\subset
\scF^1\subset
\scF^2=\bR^2, 
\]
(of affine subspaces with $\dim\scF^i=i$, we also fix 
orientations of subspaces) 
satisfying the following conditions: 
\begin{itemize}
\item[(i)] (genericity) $\scF^0$ is not contained in 
$\bigcup_{i=1}^n H_i$, and $\scF^1$ intersects with 
$\bigcup_{i=1}^n H_i$ at distinct $n$ points. 
\item[(ii)] (near to $\infty$) 
\begin{itemize}
\item[--] 
$\scF^0$ does not separate $n$ points $\scF^1\cap H_i$ ($i=1, \dots, n$) 
in $\scF^1$. 
\item[--] 
$\scF^1$ does not separate intersections of $\A$ in $\bR^2$. 
\end{itemize}
\end{itemize}
(See Figure \ref{fig:example} for the example.) 
Each line $H_i$ determines two half spaces $H_i^\pm$. 
We choose $H_i^\pm$ so that 
$\scF^0\in H_i^-$ for all $i=1, \dots, n$. 
We also fix an orientation of $\scF^1$ and after 
re-numbering the lines, if necessary, we may 
assume the following 
\[
\scF^0<H_1\cap\scF^1<H_2\cap\scF^1<\dots<H_n\cap\scF^1, 
\]
with respect to the ordering of $\scF^1$. 

Associate to such flag $\scF=\{\scF^\bullet\}$, we define a 
subset of $\ch(\A)$ as follows. 
\[
\ch_\scF^i(\A)=
\{C\in\ch(\A)\mid C\cap\scF^{i-1}=\emptyset, C\cap\scF^{i-1}\neq\emptyset\}. 
\]
We denote by $R[\ch_\scF^i(\A)]=
\bigoplus_{C\in\ch_\scF^i(\A)}R\cdot[C]$, the 
free $R$-module generated by $C\in\ch_\scF^i(\A)$, where 
$R$ is a commutative ring. It is known that 
$\rank_R A^i_R(\A)=|\ch_\scF^i(\A)|$ (\cite{yos-lef}). 
We fix notations as follows. 

\begin{Assumption}
\label{assumpt:chambers}
Let us set 
$\ch_\scF^0(\A)=\{C_0\}, \ch_\scF^1(\A)=\{C_1, \dots, C_n\}$ and 
$\ch_\scF^2(\A)=\{D_1, D_2, \dots, D_b\}$, 
where $b=|\ch_\scF^2(\A)|$. We can choose $C_1, \dots, C_n$ 
such that $\Sep(C_0, C_i)=\{H_1, H_2, \dots, H_i\}$, or equivalently, 
$C_i=H_1^+\cap\dots\cap H_{i}^+\cap H_{i+1}^-\cap\dots\cap H_n^-$, for all 
$i=0, 1, \dots, n$ (see Figure \ref{fig:example}). 
\end{Assumption}
When $1\leq i<n$, the boundary of $C_i\cap\scF^1$ consists of 
two points, $H_i\cap\scF^1$ and $H_{i+1}\cap\scF^1$, while 
$C_n\cap\scF^1$ is a half-line and its boundary consists of a point 
$H_n\cap\scF^1$. 

\begin{Definition}
\label{def:nablaeta}
We use the notations above. 
Consider $\eta=\sum_{i=1}^na_i e_i\in A_R^1(\A)$. 
\begin{itemize}
\item[$(1)$] 
Define the $R$-homomorphisms $\nabla_\eta\colon R[\ch_\scF^0(\A)]\longrightarrow
R[\ch_\scF^{1}(\A)]$ as follows. 
\begin{equation*}
\begin{split}
\nabla_\eta([C_0])
&=\sum_{C\in\ch_\scF^1(\A)}\left(
\sum_{H_i\in\Sep(C_0,C)}a_i\right)\cdot [C]\\
&=\sum_{i=1}^n(a_1+\dots+a_i)\cdot[C_i].
\end{split}
\end{equation*}
\item[$(2)$] 
Define the map 
\[
\deg\colon\ch_\scF^1(\A)\times\ch_\scF^2(\A)\longrightarrow\{\pm 1, 0\}, 
\]
as follows. 
\begin{itemize}
\item[$(i)$] If $i<n$, then the segment $C_i\cap\scF^1$ has two boundaries, 
say, $H_i\cap\scF^1$ and $H_{i+1}\cap\scF^1$. 
\[
\deg(C_i, D)=
\left\{
\begin{array}{rl}
1&\mbox{ if $D\subset H_i^-\cap H_{i+1}^+$,}\\
-1&\mbox{ if $D\subset H_i^+\cap H_{i+1}^-$,}\\
0&\mbox{ otherwise.} 
\end{array}
\right.
\]
\item[$(ii)$] If $i=n$, 
\[
\deg(C_n, D)=
\left\{
\begin{array}{rl}
-1&\mbox{ if $D\subset H_n^+$,}\\
0&\mbox{ if $D\subset H_n^-$.}
\end{array}
\right.
\]
\end{itemize}
\item[$(3)$]
Define the $R$-homomorphisms $\nabla_\eta\colon R[\ch_\scF^1(\A)]\longrightarrow
R[\ch_\scF^{2}(\A)]$ as follows. 
\[
\nabla_\eta([C])=
\sum_{D\in\ch_\scF^2(\A)}\deg(C, D)
\left(
\sum_{H_i\in\Sep(C, D)}a_i
\right)
\cdot [D]. 
\]
\end{itemize}
\end{Definition}

\begin{Proposition}
\label{prop:cohom}
(\cite{yos-cha}) $(R[\ch_\scF^\bullet(\A)], \nabla_\eta)$ is a 
cochain complex. Furthermore, there is a natural 
isomorphism of cochain complexes, 
\begin{equation}
\label{eq:chamber}
\varphi\colon(R[\ch_\scF^\bullet(\A)], \nabla_\eta)
\stackrel{\simeq}{\longrightarrow}
(A_R^\bullet(\A), \eta). 
\end{equation}
At degree $1$, the isomorphism is explicitly given by 
\begin{equation}
\label{eq:explicit}
R[\ch_\scF^1(\A)]
\stackrel{\simeq}{\longrightarrow}
A_R^1(\A),\ 
[C_i]\longmapsto
\varphi([C_i])=
\left\{
\begin{array}{cl}
e_i-e_{i+1}&\mbox{ if }i<n,\\
e_n&\mbox{ if }i=n. 
\end{array}
\right.
\end{equation}
In particular, we have 
\[
H^1(R[\ch_\scF^\bullet(\A)], \nabla_\eta)
\simeq
H^1(A_R^\bullet(\A), \eta). 
\]
\end{Proposition}

The isomorphism (\ref{eq:chamber}) is 
natural in the sense that it respects Borel-Moore homology 
\cite{yos-cha, ito-yos, yos-str}. 
Recall that each chamber $C\in\ch_\scF^2(\A)$ (with suitable 
orientation) determines a Borel-Moore $2$-homology cycle 
$[C]\in H_2^{BM}(M(\A), R)$ of the complexified complement 
$M(\A)$. The isomorphism (\ref{eq:chamber}), for $i=2$, 
is obtained by the composition 
\begin{equation}
\label{eq:composition}
R[\ch_\scF^2(\A)]\longrightarrow H_2^{BM}(M(\A), R)
\stackrel{\simeq}{\longrightarrow}H^2(M(\A), R)\simeq
A_R^i(\A). 
\end{equation}

\begin{Example}
\label{ex:chambers}
Let $\A=\{H_1, \dots, H_6\}$ be six affine lines as in 
Figure \ref{fig:example}. We also fix a flag (with orientation) 
$\scF=\{\scF^0\subset\scF^1\}$ (as in Figure \ref{fig:example}). 
There are $16$ chambers. We have 
\begin{equation*}
\begin{split}
\ch_\scF^0(\A)&=\{C_0\}, \\
\ch_\scF^1(\A)&=\{C_1, C_2, \dots, C_6\}, \\
\ch_\scF^2(\A)&=\{D_1, D_2, \dots, D_9\}. 
\end{split}
\end{equation*}
The degree maps are computed, as follows. 
\[
\begin{array}{r|rrrrrrrrr}
\deg(C_i, D_j)&D_1&D_2&D_3&D_4&D_5&D_6&D_7&D_8&D_9\\
\hline
C_1&0&-1&0&0&1&0&1&1&0\\
C_2&-1&0&0&-1&-1&0&0&-1&0\\
C_3&1&1&1&1&1&1&0&1&1\\
C_4&-1&-1&-1&0&0&0&0&0&0\\
C_5&0&0&0&-1&-1&-1&0&0&0\\
C_6&0&0&0&0&0&0&-1&-1&-1
\end{array}
\]
Consider $\eta=\sum_{i=1}^6 a_i e_i\in A_R^1(\A)$. 
We will compute the map $\nabla_\eta$. The first one 
$\nabla_\eta\colon R[\ch_\scF^0(\A)]\longrightarrow R[\ch_\scF^1(\A)]$ 
is, by definition, 
\[
\nabla_\eta([C_0])=a_1\cdot[C_1]+
a_{12}\cdot[C_2]+
a_{123}\cdot[C_3]+
a_{1234}\cdot[C_4]+
a_{12345}\cdot[C_5]+
a_{123456}\cdot[C_6], 
\]
where $a_{ijk}=a_i+a_j+a_k$, etc. 
The second one 
$\nabla_\eta\colon R[\ch_\scF^1(\A)]\longrightarrow R[\ch_\scF^2(\A)]$ is 
given as follows. 
\begin{equation*}
\begin{split}
&\nabla_\eta
\begin{pmatrix}
[C_1]\\
[C_2]\\
\vdots\\
[C_6]
\end{pmatrix}
\\
&
=
\begin{pmatrix}
0&-a_4&0&0&a_{1245}&0&a_{123456}&a_{12456}&0\\
-a_4&0&0&-a_{45}&-a_{145}&0&0&-a_{1456}&0\\
a_{34}&a_{234}&a_{1234}&a_{345}&a_{1345}&a_{12345}&0&a_{13456}&a_{123456}\\
-a_3&-a_{23}&-a_{123}&0&0&0&0&0&0\\
0&0&0&-a_3&-a_{13}&-a_{123}&0&0&0\\
0&0&0&0&0&0&-a_{1}&-a_{13}&-a_{123}
\end{pmatrix}
\begin{pmatrix}
[D_1]\\
[D_2]\\
\vdots\\
[D_9]
\end{pmatrix}. 
\end{split}
\end{equation*}

\begin{figure}[htbp]
\begin{picture}(400,150)(0,0)
\thicklines

\put(10,25){\circle*{4}}
\put(10,30){$\scF^0$}
\multiput(0,25)(5,0){70}{\circle*{2}}
\put(350,25){\vector(1,0){0}}
\put(350,30){$\scF^1$}

\put(10,10){\line(5,3){230}}
\multiput(150,10)(30,0){2}{\line(0,1){140}}
\multiput(230,10)(60,0){3}{\line(-5,3){230}}

\put(3,0){$H_1$}
\put(145,0){$H_2$}
\put(175,0){$H_3$}
\put(228,0){$H_4$}
\put(288,0){$H_5$}
\put(348,0){$H_6$}

\thinlines
\put(70,46){\vector(1,0){15}}
\put(70,46){\vector(-1,0){15}}
\put(90,42){{\footnotesize{$H_1^+$}}}
\put(40,42){{\footnotesize{$H_1^-$}}}

\put(290,46){\vector(1,0){15}}
\put(290,46){\vector(-1,0){15}}
\put(310,42){{\footnotesize{$H_6^+$}}}
\put(260,42){{\footnotesize{$H_6^-$}}}

\thicklines

\put(30,80){$C_0$}
\put(120,40){$C_1$}
\put(160,30){$C_2$}
\put(200,5){$C_3$}
\put(210,35){$C_4$}
\put(210,71){$C_5$}
\put(270,100){$C_6$}

\put(160,65){$D_1$}
\put(132,72){$D_2$}
\put(35,140){$D_3$}

\put(163,90){$D_4$}
\put(153,108){$D_5$}
\put(95,140){$D_6$}

\put(135,140){$D_9$}
\put(160,140){$D_8$}
\put(195,140){$D_7$}

\end{picture}
      \caption{Example \ref{ex:chambers}}
\label{fig:example}
\end{figure}
\end{Example}

\subsection{Aomoto complex via resonant bands}
\label{subsec:viaresban}

Let $\A=\{H_1, \dots, H_n\}$ be an arrangement of 
affine lines in $\bR^2$. We fix 
the flag $\scF$ as in \S\ref{subsec:viacha}. 
The cohomology of the Aomoto complex can be computed using chambers. 
In this subsection, we introduce the notion ``$\eta$-resonant bands'' 
which enables us to simplify the computation of cohomology. 
This can be regarded as ``the Aomoto complex version'' 
of the results in \cite{yos-mil, yos-res}. 

\begin{Definition}
\label{def:band} 
A {\em band} $B$ is a region bounded by a pair of consecutive 
parallel lines $H_i$ and $H_{i+1}$. 
\end{Definition}
Each band $B$ contains two unbounded chambers 
$U_1(B), U_2(B)\in\ch(\A)$. Since $B$ intersects $\scF^1$, 
we may assume that $B\cap\scF^1=U_1(B)\cap\scF^1$ and 
$U_2(B)\cap\scF^1=\emptyset$. In other words, 
$U_1(B)\in\ch_\scF^1(\A)$ and $U_2(B)\in\ch_\scF^2(\A)$. 
The distance $d(U_1(B), U_2(B))$ is called the \emph{length} of 
the band $B$.

\begin{Definition}
Let $\eta=\sum_{i=1}^n a_i e_i\in A_R^1(\A)$. 
A band $B$ is called \emph{$\eta$-resonant} if 
\[
\sum_{H_i\in\Sep(U_1(B), U_2(B))}a_i=0. 
\]
We denote by $\RB_\eta(\A)$ the set of all $\eta$-resonant bands. 
\end{Definition}
We can extend $U_1$ to an injective $R$-module homomorphism 
$U_1\colon R[\RB_\eta(\A)]\hookrightarrow R[\ch_\scF^1(\A)]$. We denote 
by $\widetilde{\nabla}_\eta:=-\nabla_\eta\circ U_1\colon R[\RB_\eta(\A)]
\longrightarrow R[\ch_\scF^2(\A)]$ the composition of 
$U_1$ and $\nabla_\eta$ (multiplied by $-1$). 
More precisely, to each 
$\eta$-resonant band $B\in\RB_\eta(\A)$, we associate 
an element $\widetilde{\nabla}_\eta(B)\in R[\ch_\scF^2(\A)]$ 
as follows. 
\[
\widetilde{\nabla}_\eta(B):=-\nabla_\eta(U_1(B))=
\sum_{D\in\ch(\A), D\subset B}
\left(
\sum_{H_i\in\Sep(U_1(B), D)}a_i
\right)\cdot[D]. 
\]

\begin{Example}
Let $\A=\{H_1, \dots, H_6\}$ be lines as in Figure \ref{fig:example}. 
There are three bands $B_1, B_2, B_3$, i.e., those defined by 
$(H_2, H_3)$, $(H_4, H_5)$ and $(H_5, H_6)$, respectively. 
We have 
$U_1(B_1)=C_2, U_2(B_1)=D_8$, 
$U_1(B_2)=C_4, U_2(B_2)=D_3$ and 
$U_1(B_3)=C_5, U_2(B_3)=D_6$. 
The band $B_1$ has length $4$, while $B_2$ and $B_3$ have length $3$. 
Let $\eta=a_1e_1+\dots+a_6e_6\in A_R^1(\A)$. 
The band $B_1$ is $\eta$-resonant if and only if 
$a_1+a_4+a_5+a_6=0$. Then we have 
\[
\widetilde{\nabla}_\eta([B_1])=
a_4[D_1]+(a_4+a_5)[D_4]+
(a_1+a_4+a_5)[D_5]. 
\]
\end{Example}

Obviously the map $U_1$ induces $U_1\colon\Ker(\widetilde{\nabla}_\eta)
\longrightarrow\Ker({\nabla}_\eta:R[\ch_\scF^1(\A)]\longrightarrow
R[\ch_\scF^2(\A)])$. Thus we have a natural map 
\begin{equation}
\label{eq:naturalmap}
\widetilde{U}_1\colon\Ker(\widetilde{\nabla}_\eta)\longrightarrow
H^1(R[\ch_\scF^\bullet(\A)], \nabla_\eta). 
\end{equation}

The above map $\widetilde{U}_1$ is neither injective nor surjective 
in general. 
The following is the main result concerning resonant bands which 
asserts that the map $\widetilde{U}_1$ above (\ref{eq:naturalmap}) is 
isomorphic under certain non-resonant assumption at infinity. 
This provides 
an effective way to compute 
$H^1(R[\ch_\scF^\bullet(\A)], \nabla_\eta)$. Indeed, normally, 
$|\RB_\eta(\A)|$ is much smaller than $|\ch_\scF^1(\A)|$. 

\begin{Theorem}
\label{thm:resbanAomoto}
Let $R$ be a commutative ring and 
$\eta=\sum_{i=1}^n a_i e_i\in A_R^1(\A)$. 
\begin{itemize}
\item[(i)] 
Suppose that 
$\alpha:=\sum_{i=1}^n a_i\in R^\times$ is invertible. Then 
the natural map $\widetilde{U}_1$ injective. 
\item[(ii)] 
We assume that $R$ is an integral domain and 
$\alpha:=\sum_{i=1}^n a_i\in R^\times$. Then $\widetilde{U}_1$ is 
isomorphic. 
\item[(iii)] 
Let $R$ be an arbitrary commutative ring. 
If $\alpha:=\sum_{i=1}^n a_i\in R^\times$ and 
all bands are $\eta$-resonant, then the natural 
map $\widetilde{U}_1$ is isomorphic. 
\end{itemize}
\end{Theorem}
\begin{proof}

(i) 
Let $\sum_{B\in\RB_\eta(\A)}r_B\cdot [B]\in R[\RB_\eta(\A)]$, ($r_B\in R$). 
Suppose $\sum r_B\cdot[B]\in\Ker\widetilde{U}_1$, that is, 
$\widetilde{U}_1(\sum r_B\cdot[B])\in\image
\left(\nabla_\eta\colon R[\ch_\scF^0(\A)]\longrightarrow R[\ch_\scF^1(\A)]\right)$. 
Since $R[\ch_\scF^0(\A)]=R\cdot [C_0]$, there exists an element 
$s\in R$ such that 
\begin{equation}
\label{eq:compare}
\sum r_B\cdot [U_1(B)]=
s\cdot\nabla_\eta([C_0]). 
\end{equation}
Note that in the left hand side of (\ref{eq:compare}), 
the chamber $C_n$ does not appear, because $C_n$ is not bounded 
by two parallel lines. 
By Definition \ref{def:nablaeta}, $\nabla_\eta([C_0])=\sum_{i=1}^n
(a_1+\dots+a_i)\cdot[C_i]$. The coefficient of $[C_n]$ is equal to 
$s\cdot(a_1+\dots+a_n)=s\cdot\alpha$. By the assumption that 
$\alpha$ is invertible, we have $s=0$. 
Hence $\sum r_B\cdot U_1(B)=s\cdot\nabla_\eta([C_0])=0$, and we have 
$\sum r_B\cdot[B]=0$. 

Next we show the surjectivity of (\ref{eq:naturalmap}). 
Suppose that 
$\beta=\sum_{i=1}^n b_i\cdot[C_i]\in\Ker(\nabla_\eta\colon
R[\ch_\scF^1(\A)]\longrightarrow R[\ch_\scF^2(\A)])$. 
Consider the following element, 
\begin{equation}
\begin{split}
\beta'
&=\beta-\frac{b_n}{\alpha}\cdot\nabla_\eta([C_0])\\
&=\sum_{i=1}^{n-1}b_i'\cdot[C_i]. 
\end{split}
\end{equation}
Obviously, $\beta$ and $\beta'$ represent the same element in 
$H^1(R[\ch_\scF^\bullet(\A)], \nabla_\eta)$. It is sufficient 
to show $\beta'\in \image \widetilde{U}_1$. 

Next we consider the chamber $C_i$ ($i<n$) such that $H_i$ and 
$H_{i+1}$ are not parallel. Then there is a unique chamber 
$D_p\in\ch_\scF^2(\A)$ 
such that $\Sep(C_i, D_p)=\A$, which is called the 
``opposite chamber of $C_i$'' in \cite[Def. 2.1]{yos-bas} and denoted by 
$D_p=C_i^\lor$. 
Then we consider the coefficient $c_{D_p}$ of 
$[D_p]$ 
in $\nabla_\eta(\beta')=\sum c_D\cdot[D]$. Since $D_p=C_i^\lor$ 
appears only in $\nabla_\eta([C_i])$ and $\nabla_\eta([C_n])$, and 
the coefficient of $[C_n]$ is already zero, we have 
$c_{D_p}=\alpha\cdot b_i'$. By the assumption that $\alpha\in R^\times$, 
$\nabla_\eta(\beta')=0$, in particular $c_{D_p}=0$, implies that $b_i'=0$. 
So $\beta'=\sum_{i=1}^{n-1}b_i'\cdot[C_i]$ is a linear combination of 
$C_i$'s ($i<n$) such that $H_i$ and $H_{i+1}$ are parallel. 
So far, we use only the fact $\alpha\in R^\times$. If all bands 
are $\eta$-resonant, then we have already proved that $\beta'$ is 
generated by $U_1(B)$ with $B\in\RB_\eta(\A)$. 
Thus (iii) is proved. 

Now we assume that $R$ is an integral domain. We will prove (ii). 
Let $C_i$ be a chamber such that walls $H_i$ and $H_{i+1}$ are parallel. 
Let $B$ be the corresponding band 
defined by $H_i$ and $H_{i+1}$. Note that $C_i=U_1(B)$ and its opposite 
chamber is $U_2(B)$. 
Suppose that $B$ is not an $\eta$-resonant 
band, that is, $\alpha':=\sum_{H_j\in\Sep(U_1(B),U_2(B))}a_j\neq 0$. 
Again consider the coefficient of $[U_2(B)]$ in $\nabla_\eta(\beta')$. 
Since $[U_2(B)]$ appears in $\nabla_\eta([C_i])$ and other terms 
$\nabla_\eta([C_k])$ for $k$ 
such that $H_k$ and $H_{k+1}$ are not parallel. However the coefficients 
of chambers of the second type in $\beta'$ are already zero. 
Therefore the coefficient of $[U_2(B)]$ in $\nabla_\eta(\beta')$ 
is $-\alpha'\cdot b_i'$, which is zero. Since $R$ is an integral domain, 
we have $b_i'=0$. Hence $\beta'$ is a linear combination of 
$U_1(B)$'s where $B\in\RB_\eta(\A)$. This completes the proof of 
the surjectivity. 
\end{proof}

\begin{Remark}
\label{rem:bandtoOS}
Equation (\ref{eq:naturalmap}) and Theorem \ref{thm:resbanAomoto} are 
concerning the following homomorphism of cochains. 
\[
\begin{CD}
@. 0 @>>> R[\RB_\eta(\A)] @>\widetilde{\nabla}_\eta>> R[\ch(\A)] @>>> 0\\
@. @VVV @VV\varphi_1 V @VV\varphi_2 V\\
0@>>> A_R^0(\A) @>\eta>> A_R^1(\A) @>\eta>> A_R^2(\A) @>>> 0
\end{CD}
\]
The map $\widetilde{U}_1$ is nothing but the homomorphism 
$\Ker(\widetilde{\nabla}_\eta)\longrightarrow H^1(A_R^\bullet(\A), \eta)$ 
induced from $\varphi_1$. 
By Proposition \ref{prop:cohom} (especially, the explicit map 
(\ref{eq:explicit})), the map $\varphi_1$ above 
is given by 
\[
[B]\longmapsto e_i-e_{i+1}, 
\]
where $B$ is a $\eta$-resonant band bounded by 
the lines $H_i$ and $H_{i+1}$. 
\end{Remark}

\begin{Example}
\label{ex:depth2}
Let $R=\bF_2$. Let $\A=\{H_1, \dots, H_6\}$ be an 
arrangement of affine lines as in Figure \ref{fig:depth2} (which is 
$\A(7,1)$ in \cite{gru}). 
Let $\eta=e_2+e_3+e_6\in A_R^1(\A)$ (the supporting lines of 
$\eta$ are colored blue). 
There are three 
bands 
$B_1$ (bounded by $H_1$ and $H_2$), 
$B_2$ (bounded by $H_3$ and $H_4$),  and 
$B_3$ (bounded by $H_5$ and $H_6$). 
$\Sep(U_1(B_1), U_2(B_1))=\{H_3, H_4, H_5, H_6\}$ and 
two lines of them, $H_3$ and $H_5$, have non zero coefficient in $\eta$. 
Hence $B_1$ is an $\eta$-resonant band. Similarly, we have 
$\RB_\eta(\A)=\{B_1, B_2, B_3\}$. 
By definition, 
$\widetilde{\nabla}_\eta(B_1)=
\widetilde{\nabla}_\eta(B_2)=
\widetilde{\nabla}_\eta(B_3)=[D_1]$. Hence the kernel 
\[
\Ker(\widetilde{\nabla}_\eta\colon
\bF_2[\RB_\eta(\A)]\longrightarrow
\bF_2[\ch(\A)])
\]
is $2$-dimensional (generated by 
$[B_1]-[B_2]$ and $[B_2]-[B_3]$). By 
Theorem \ref{thm:resbanAomoto}, 
$H^1(A_{\bF_2}^\bullet(\A), \eta)\simeq\bF_2^2$. 

\begin{figure}[htbp]
\begin{picture}(300,150)(0,0)
\thicklines

\put(10,10){\line(5,3){230}}
\color{blue}
\put(70,10){\line(5,3){230}}
\normalcolor

\color{blue}
\put(150,10){\line(0,1){140}}
\normalcolor
\put(180,10){\line(0,1){140}}

\put(230,10){\line(-5,3){230}}
\color{blue}
\put(290,10){\line(-5,3){230}}
\normalcolor

\put(3,0){$H_1$}
\put(63,0){$H_2$}
\put(145,0){$H_3$}
\put(175,0){$H_4$}
\put(228,0){$H_5$}
\put(288,0){$H_6$}

{\scriptsize 
\put(65,30){$U_1(B_1)$}
\put(215,120){$U_2(B_1)$}
\put(151,30){$U_1(B_2)$}
\put(151,120){$U_1(B_2)$}
\put(210,30){$U_1(B_3)$}
\put(60,120){$U_2(B_3)$}
}

\put(152,72){$D_1$}
\put(132,72){$D_2$}

\put(163,52){$D_3$}
\put(163,90){$D_4$}

\end{picture}
      \caption{Example \ref{ex:depth2}}
\label{fig:depth2}
\end{figure}

\end{Example}

\begin{Example}
\label{ex:A(16,1)}
We consider 
$\overline{\A}=\A(16,1)=\{\overline{H}_1, \dots, \overline{H}_{16}\}$ 
from the Gr\"unbaum's catalogue \cite{gru}, 
see Figure \ref{fig:grunbaum}. Let us denote by 
$\A=\{H_2, H_3, \dots, H_{16}\}$ the deconing 
$\dec_{\widetilde{H}_1}\widetilde{\A}$, 
the lower-left one in Figure \ref{fig:grunbaum}. 
The affine arrangement $\A$ has $7$ bands $B_1, \dots, B_7$. 
To indicate the choice of $U_1(B)$ and $U_2(B)$, 
we always put the name $B$ of the band in the unbounded chamber $U_1(B)$. 

Let $R=\Z/8\Z$. Define $\widetilde{\eta}_1, \widetilde{\eta}_2
\in A_R^1(\widetilde{\A})_0$ by 
\[
\begin{split}
\widetilde{\eta}_1&=
\widetilde{e}_1+
\widetilde{e}_3+
\widetilde{e}_5+
\widetilde{e}_7+
\widetilde{e}_9+
\widetilde{e}_{11}+
\widetilde{e}_{13}+
\widetilde{e}_{15},\\
\widetilde{\eta}_2&=
\widetilde{e}_2+
\widetilde{e}_4+
\widetilde{e}_6+
\widetilde{e}_8+
\widetilde{e}_{10}+
\widetilde{e}_{12}+
\widetilde{e}_{14}+
\widetilde{e}_{16}, 
\end{split}
\]
and set $\widetilde{\eta}:=\widetilde{\eta}_1+6\widetilde{\eta}_2$. 

Let $\eta=(e_3+e_5+e_7+\dots+e_{15})+6(e_2+e_4+\dots+e_{16})
\in A_R^1(\A)$. 
Then all $7$ bands are $\eta$-resonant. Thus we can apply theorem 
Theorem \ref{thm:resbanAomoto} (iii). The kernel 
$\Ker(\widetilde{\nabla}_\eta\colon R[\RB_\eta(\A)]\longrightarrow 
R[\ch(\A)])$ is a free $R$-module generated by 
\[
[B_1]+2[B_2]+3[B_3]+4[B_4]+5[B_5]+6[B_6]+7[B_7]. 
\]
The corresponding element (via the correspondence Remark 
\ref{rem:bandtoOS}) in $A_R^1(\widetilde{\A})_0$ is 
\[
4(\widetilde{e}_{2}+\widetilde{e}_{3})+
3(\widetilde{e}_{4}-\widetilde{e}_{7}+\widetilde{e}_{13}-
\widetilde{e}_{16})+
2(\widetilde{e}_{6}+\widetilde{e}_{9}-\widetilde{e}_{11}
-\widetilde{e}_{14})+(\widetilde{e}_{5}+
\widetilde{e}_{8}-\widetilde{e}_{12}-\widetilde{e}_{15}). 
\]
By Theorem \ref{thm:resbanAomoto} (iii), the cohomology of the
Aomoto complex 
\[
H^1(A_R^1(\widetilde{\A})_0, \widetilde{\eta})\simeq
H^1(A_R^1(\A), \eta)\simeq\Ker(\widetilde{\nabla}_\eta)
\simeq R\simeq \Z/8\Z 
\]
is non-vanishing. 

\begin{figure}[htbp]
\begin{picture}(360,370)(0,-10)



\normalcolor
\put(0,90){\line(1,0){160}}

\color{blue}
\put(0,80){\line(1,0){160}}

\normalcolor
\qbezier(0,123.1371)(80,90)(160,56.8629)

\color{blue}
\qbezier(71.0051,0)(108.2843,90)(137.227,160)

\normalcolor
\put(10,160){\line(1,-1){150}}

\color{blue}
\qbezier(0,111.421)(80,78.2843)(160,45.1472)

\normalcolor
\qbezier(117.2792,0)(80,90)(51.0571,160)

\color{blue}
\put(10,0){\line(1,1){150}}

\normalcolor
\put(80,0){\line(0,1){160}}

\color{blue}
\put(0,150){\line(1,-1){150}}

\normalcolor
\qbezier(42.72,0)(80,90)(108.495,160)

\color{blue}
\qbezier(160,111.421)(80,78.2843)(0,45.1472)

\normalcolor
\put(0,10){\line(1,1){150}}

\color{blue}
\qbezier(88.9949,0)(51.7157,90)(22.7728,160)

\normalcolor
\qbezier(160,123.1371)(80,90)(0,56.8629)

{\footnotesize 
\put(-7,75){$3$}
\put(-7,88){$2$}
\put(-7,109){$7$}
\put(-7,122){$4$}
\put(-10,40){$13$}
\put(-10,53){$16$}
\put(-10,148){$11$}
\put(4,162){$6$}
\put(18,162){$15$}
\put(48,162){$8$}
\put(75,162){$10$}
\put(105,162){$12$}
\put(133,162){$5$}
\put(145,162){$14$}
\put(159,152){$9$}
}

{\scriptsize 
\put(35,150){$B_1$}
\put(10,142){$B_2$}
\put(10,108){$B_3$}
\put(10,82){$B_4$}
\put(10,54.5){$B_5$}
\put(10,12){$B_6$}
\put(58,12){$B_7$}

}

\put(50,-15){$\dec_{\widetilde{H}_{1}}\widetilde{\A}$}


\color{blue}
\put(200,160){\line(1,-1){160}}
\normalcolor
\put(280,0){\line(0,1){160}}
\color{blue}
\put(200,0){\line(1,1){160}}
\normalcolor
\put(273.3726,0){\line(0,1){160}}
\color{blue}
\qbezier(360,12.1177)(280,68.6863)(200,125.2548)
\normalcolor
\put(264,0){\line(0,1){160}}
\color{blue}
\qbezier(200,34.7452)(280,91.3138)(360,147.8823)
\normalcolor
\put(241.3726,0){\line(0,1){160}}
\color{blue}
\put(200,64){\line(1,0){160}}


\color{blue}
\put(200,96){\line(1,0){160}}
\normalcolor
\put(318.6274,0){\line(0,1){160}}
\color{blue}
\qbezier(200,12.1177)(280,68.6863)(360,125.2548)
\normalcolor
\put(296,0){\line(0,1){160}}
\color{blue}
\qbezier(360,34.7452)(280,91.3138)(200,147.8823)
\normalcolor
\put(286.6274,0){\line(0,1){160}}

{\footnotesize 
\put(197,162){$1$}
\put(239,162){$8$}
\put(261,162){$6$}
\put(270,162){$4$}
\put(277,162){$2$}
\put(283,162){{\tiny $16$}}
\put(292,162){$14$}
\put(315,162){$12$}
\put(360,162){$3$}
\put(362,145){$7$}
\put(361,121){$13$}
\put(361,91){$11$}
\put(361,61){$9$}
\put(361,31){$15$}
\put(361,8){$5$}

\put(210,123){$B'_1$}
\put(210,77){$B'_2$}
\put(210,30){$B'_3$}
\put(245,140){$B'_4$}
\put(266,140){{\tiny $5'$}}
\put(273,140){{\tiny $6'$}}
\put(280,140){{\tiny $7'$}}
\put(288,140){{\tiny $8'$}}
\put(302,140){$B'_9$}

}

\put(250,-15){$\dec_{\widetilde{H}_{10}}\widetilde{\A}$}


\color{blue}
\put(100,254){\line(1,0){160}}
\normalcolor
\put(100,270){\line(1,0){160}}
\color{blue}
\put(100,286){\line(1,0){160}}
\normalcolor
\qbezier(260,236.8629)(180,270)(100,303.1371)
\color{blue}
\put(100,327.3726){\line(1,-1){137.3726}}
\normalcolor
\put(100,350){\line(1,-1){160}}
\color{blue}
\put(122.6274,350){\line(1,-1){137.3726}}
\normalcolor
\qbezier(146.8629,350)(180,270)(213.1371,190)
\color{blue}
\put(164,190){\line(0,1){160}}
\normalcolor
\put(180,190){\line(0,1){160}}
\color{blue}
\put(196,190){\line(0,1){160}}
\normalcolor
\qbezier(146.8629,190)(180,270)(213.1371,350)
\color{blue}
\put(100,212.6274){\line(1,1){137.3726}}
\normalcolor
\put(100,190){\line(1,1){160}}
\color{blue}
\put(122.6274,190){\line(1,1){137.3726}}
\normalcolor
\qbezier(100,236.8629)(180,270)(260,303.1371)

{\small
\put(94,250){$1$}
\put(94,267){$2$}
\put(94,282){$3$}
\put(94,300){$4$}
\put(94,324){$5$}
\put(94,352){$6$}
\put(120,352){$7$}
\put(142,352){$8$}
\put(162,352){$9$}
\put(175,352){$10$}
\put(190,352){$11$}
\put(207,352){$12$}
\put(235,352){$13$}
\put(262,352){$14$}
\put(262,324){$15$}
\put(262,300){$16$}
}

\put(155,179){$\overline{\A}=\A(16,1)$}

\end{picture}
      \caption{$\A(16,1)$ and deconings with respect to 
$H_1$ and $H_{10}$.}
\label{fig:grunbaum}
\end{figure}
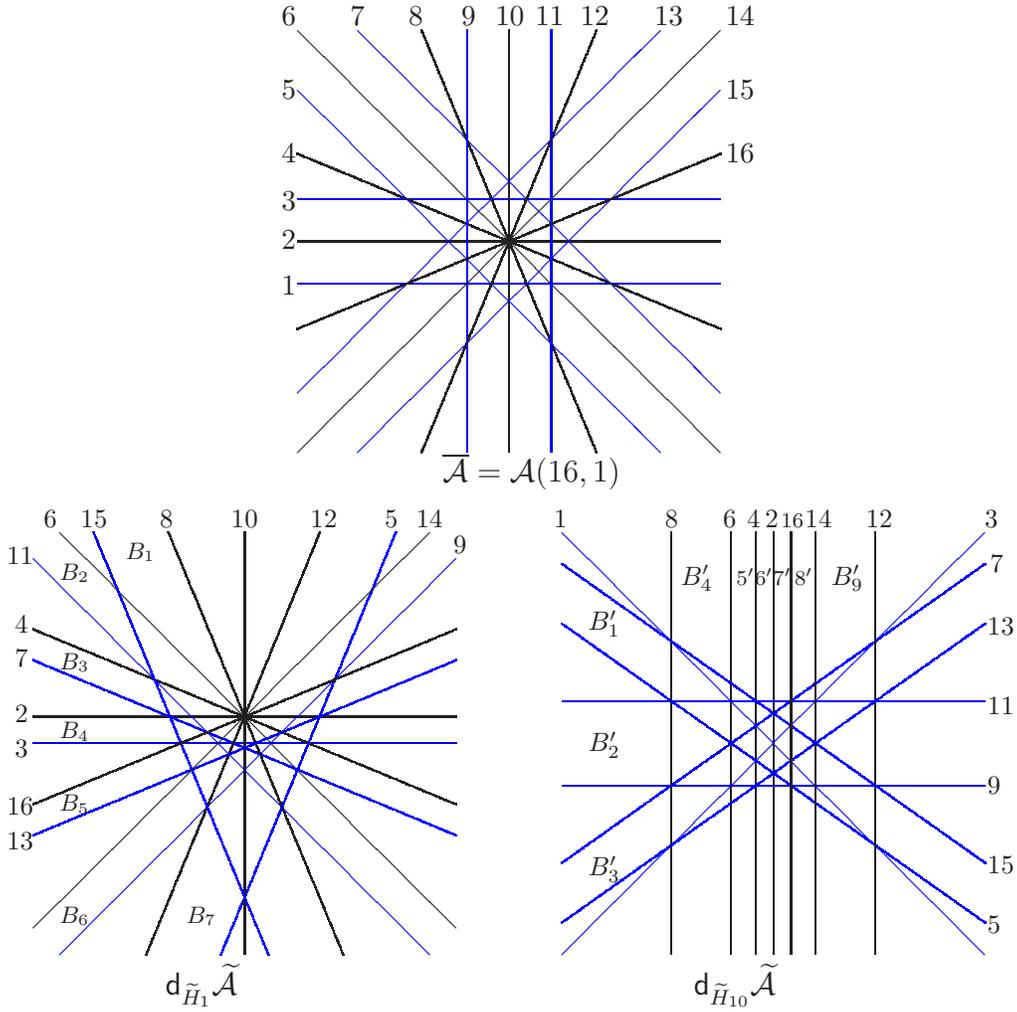

\end{Example}

\begin{Remark}
\label{rem:8torsion}
Let us point out a possible relation between 
$\bZ/8\bZ$-resonance in Example \ref{ex:A(16,1)} and 
isolated torsion points of order $8$ in the characteristic variety 
of $\A(16,1)$. 
Let us denote $M=M(\A(16,1))=\C\bP^2\setminus\bigcup_{H\in
\A(16,1)}H_\C$ the complexified complement. 
Recall that the character torus of $M$ is 
$\bT:=\Hom(\pi_1(M),\bC^\times)\simeq
\{\bm{t}=(t_1, t_2, \dots, t_{16})\in(\bC^\times)^{16}\mid
\prod_{i=1}^{16} t_i=1\}$. We also define the essential open 
subset of $\bT$ by 
\[
\bT^\circ:=\{\bm{t}=(t_1, \dots, t_{16})\in\bT\mid 
t_i\neq 1, \forall i=1, \dots, 16\}. 
\]
The characteristic variety $\scV^1(\A(16,1))$ 
of $\A(16,1)$ is the set of 
points in the character torus $\bT$ such that 
the associated local system has non-vanishing first cohomology, 
i.e., 
\[
\scV^1(\A(16,1))=\{\bm{t}\in\bT\mid
\dim H^1(M, \scL_{\bm{t}})\geq 1\}. 
\]
Let $\zeta=e^{2\pi i/8}$ and 
consider the following point, 
\[
\rho=(
\zeta, \zeta^6, 
\zeta, \zeta^6, 
\zeta, \zeta^6, 
\zeta, \zeta^6, 
\zeta, \zeta^6, 
\zeta, \zeta^6, 
\zeta, \zeta^6, 
\zeta, \zeta^6)\in\bT^\circ. 
\]
Let us recall quickly the resonant band algorithm for computing 
local system cohomology groups (see \cite{yos-res} for details). 
For a given local system $\scL_{\bm{t}}$, 
we define the set  $\RB_{\scL_{\bm{t}}}(\A)$ of 
$\scL_{\bm{t}}$-resonant bands and the map 
$\nabla_{\scL_{\bm{t}}}\colon
\bC[\RB_{\scL_{\bm{t}}}(\A)]\longrightarrow\bC[\ch(\A)]$. 
If $\scL_{\bm{t}}$ has non-trivial monodromy around the line 
at infinity, then we have the isomorphism 
$H^1(M, \scL_{\bm{t}})\simeq\Ker(\nabla_{\scL_{\bm{t}}})$. 

Since $\scL_\rho$ defined above has non trivial monodromy around 
any line, we can apply resonant band algorithm to any deconings. 
Here we exhibit two cases (although the results coincide 
logically), 
$\dec_{\widetilde{H}_1}\widetilde{\A}$ and 
$\dec_{\widetilde{H}_{10}}\widetilde{\A}$. 
(See Figure \ref{fig:grunbaum}.)
\begin{itemize}
\item The affine arrangement 
$\dec_{\widetilde{H}_1}\widetilde{\A}$ has seven 
bands $B_1, \dots, B_7$, which are all $\scL_\rho$-resonant. 
Then $\Ker(\nabla_{\scL_\rho})$ is one dimensional and 
generated by the following element, 
\begin{multline*}
\sin\left(\frac{\pi}{8}\right)[B_1]
-\sin\left(\frac{\pi}{4}\right)[B_2]
+\sin\left(\frac{3\pi}{8}\right)[B_3]
-\sin\left(\frac{\pi}{2}\right)[B_4]\\
+\sin\left(\frac{3\pi}{8}\right)[B_5]
-\sin\left(\frac{\pi}{4}\right)[B_6]
+\sin\left(\frac{\pi}{8}\right)[B_7]. 
\end{multline*}
\item The affine arrangement 
$\dec_{\widetilde{H}_{10}}\widetilde{\A}$ has nine 
bands $B'_1, \dots, B'_9$, which are all $\scL_\rho$-resonant. 
Then $\Ker(\nabla_{\scL_\rho})$ is one dimensional and 
generated by the following element, 
\begin{multline*}
[B_1]+\sqrt{2}[B_2]+[B_3]\\
-[B_4]+(1+\sqrt{2})[B_5]
-(2+\sqrt{2})[B_6]\\
+(2+\sqrt{2})[B_7]
+(1+\sqrt{2})[B_8]
+[B_9]. 
\end{multline*}
\end{itemize}
Hence we have $\dim H^1(M, \scL_{\rho})=1$. Furthermore, we 
can prove that $\rho$ generates the essential part of 
the characteristic variety. More precisely, we have the following, 
\begin{equation}
\scV^1(\A(16,1))\cap\bT^\circ=
\{\rho, \rho^2, \rho^3, \rho^5, \rho^6, \rho^7\}. 
\end{equation}
\end{Remark}

\subsection{Resonant bands over $\bF_2$ and subarrangements}
\label{subsec:subarr}

Let $\A=\{H_1, \dots, H_n\}$ be an arrangement of affine 
lines in $\bR^2$. Let $\scS\subset\A$ be a subset. 
Denote $e(\scS):=\sum_{H_i\in\scS}e_i\in A_{\bF_2}^1(\A)$. 
Clearly, $e(\scS)+e(\A)=e(\A\setminus\scS)$. 
Below is the summary of ``subarrangement description of 
resonant band algorithm'': 
\begin{itemize}
\item[(a)] 
Let $B$ be a band of $\A$. Then $B\in\RB_{e(\scS)}(\A)$ 
if and only if 
the number of lines in $\scS$ separating $U_1(B)$ and 
$U_2(B)$ is even, i.e., $2|\sharp(\scS\cap\Sep(U_1(B), U_2(B)))$. 
\item[(b)] 
$\widetilde{\nabla}_{e(\scS)}\colon\bF_2[\RB_{e(\scS)}(\A)]\longrightarrow
\bF_2[\ch_\scF^2(\A)]$ is given by the following formula. 
\[
\widetilde{\nabla}_{e(\scS)}(B)=
\sum_{C\in\ch(\A), C\subset B}
\left|
\scS\cap\Sep(U_1(B), C)
\right|
\cdot
[C]. 
\]
(See Example \ref{ex:depth2}). In particular, 
if we consider $\eta_0=e(\A)=
e_1+e_2+\dots+e_n$, then we have 
\[
\widetilde{\nabla}_{\eta_0}(B)=
\sum_{C\in\ch(\A), C\subset B}
d(U_1(B), C)
\cdot
[C]. 
\]
\item[(c)] 
Suppose that $|\scS|$ is odd. Then we can apply 
Theorem \ref{thm:resbanAomoto}, and we have an isomorphism 
\[
\Psi\colon
\Ker(\widetilde{\nabla}_{e(\scS)}
)
\stackrel{\simeq}{\longrightarrow}
H^1(A_{\bF_2}^\bullet(\A), e(\scS)). 
\]
\item[(d)] 
Using Remark \ref{rem:bandtoOS} (and 
Proposition \ref{prop:cohom} (especially, the explicit map 
(\ref{eq:explicit}))), the above isomorphism is 
given by 
\[
\Psi\colon
[B]\longmapsto e_i+e_{i+1}\in A_{\bF_2}^1(\A), 
\]
where $B$ is a $e(\scS)$-resonant band determined by 
the lines $H_i$ and $H_{i+1}$. 
\end{itemize}

\section{Non-existence of real $4$-nets}
\label{sec:4net}

\subsection{Aomoto complex for diagonal element}
\label{subsec:diag}

Let $\A=\{H_1, \dots, H_n\}$ be an arrangement of affine lines 
in $\R^2$ with odd $n$. 
Let 
$\widetilde{\A}=\{\widetilde{H}_0, \widetilde{H}_1, \dots, 
\widetilde{H}_n\}$ be the coning of $\A$ and 
$\overline{\A}=\{\overline{H}_0, \overline{H}_1, \dots, 
\overline{H}_n\}$ be the projectivization. 
Recall that $\widetilde{\eta}_0:=\widetilde{e}(\widetilde{\A})=
\widetilde{e}_0+\widetilde{e}_1+\dots+\widetilde{e}_n\in
A_{\bF_2}^1(\widetilde{\A})_0$ is the diagonal element and 
$\eta_0=e(\A)=e_1+\dots+e_n\in A_{\bF_2}^1(\A)$. 

Choose a subset $\widetilde{\scS}\subset\widetilde{\A}$. 
In the figures below, the lines in $\widetilde{\scS}$ are 
colored in red. The other lines are black.

As we saw in Theorem \ref{thm:mod2ps}, 
the relation $\widetilde{\eta}_0\wedge\widetilde{e}(\widetilde{\scS})=0$ 
is equivalent to 
``$|\overline{\A}_X|$ is even $\Longrightarrow |\overline{\scS}_X|$ is even'' and 
``$|\overline{\A}_X|$ is odd $\Longrightarrow$ either 
$\overline{\scS}_X=\emptyset$ or $\overline{\scS}_X=\overline{\A}_X$'' 
for $\forall X\in L_2(\overline{\A})$. 
From this, it is easily seen that if the multiplicity is 
$|\overline{\A}_X|\leq 3$, then $\overline{\A}_X$ is monocolor (either 
all red $\overline{\scS}_X=\overline{\A}_X$ or 
all black $\overline{\scS}_X=\emptyset$). 
However, when $|\overline{\A}_X|=4$, then there are four cases 
(Figure \ref{fig:4cases}): 
\begin{itemize}
\item[$(i)$] $\overline{\scS}_X=\emptyset$. 
\item[$(ii)$] $\overline{\scS}_X=\overline{\A}_X$. 
\item[$(iii)$] $|\overline{\scS}_X|=2$ and lines in 
$\overline{\scS}_X$ are adjacent. 
\item[$(iv)$] $|\overline{\scS}_X|=2$ and lines in 
$\overline{\scS}_X$ are separated by 
lines in $\overline{\A}_X\setminus\widetilde{\scS}_X$. 
\end{itemize}

\begin{figure}[htbp]
\begin{picture}(400,80)(0,0)
\thicklines

\put(10,10){\line(1,1){70}}
\put(10,45){\line(1,0){70}}
\put(45,10){\line(0,1){70}}
\put(80,10){\line(-1,1){70}}

\color{red}
\put(100,10){\line(1,1){70}}
\put(100,45){\line(1,0){70}}
\put(135,10){\line(0,1){70}}
\put(170,10){\line(-1,1){70}}
\normalcolor
\color{red}
\put(190,10){\line(1,1){70}}
\put(190,45){\line(1,0){70}}
\normalcolor
\put(225,10){\line(0,1){70}}
\put(260,10){\line(-1,1){70}}

\color{red}
\put(280,10){\line(1,1){70}}
\put(350,10){\line(-1,1){70}}
\normalcolor
\put(280,45){\line(1,0){70}}
\put(315,10){\line(0,1){70}}

\put(315,45){\circle*{4}}
\put(306,55){{\small $X$}}

\put(35,-5){$(i)$}
\put(125,-5){$(ii)$}
\put(215,-5){$(iii)$}
\put(305,-5){$(iv)$}

{\small 

\color{red}
\put(184,20){$\overline{H}_3$}
\put(185,48){$\overline{H}_1$}
\normalcolor
\put(250,20){$\overline{H}_2$}
\put(226,10){$\overline{H}_0$}

\color{red}
\put(274,20){$\overline{H}_1$}
\put(340,20){$\overline{H}_3$}
\normalcolor
\put(275,48){$\overline{H}_0$}
\put(316,10){$\overline{H}_2$}
}
\end{picture}
      \caption{Local structures of $\overline{\scS}_X$. 
(Members of $\overline{\scS}_X$ are red, and 
$\overline{\scS}_X=\{
\overline{H}_1, \overline{H}_3\}$ in $(iii)$ and $(iv)$).}
\label{fig:4cases}
\end{figure}
The cases $(iii)$ and $(iv)$ are combinatorially identical. 
However, the real structures are different. This difference 
is crucial, actually, by using resonant bands, 
we can prove that $(iv)$ can not happen 
(``Non Separation Theorem''). 

\begin{Theorem}
\label{thm:nonsep}
Let $\overline{\scS}\subset\overline{\A}$. 
Suppose that $\widetilde{\eta}_0\wedge\widetilde{e}(\widetilde{\scS})=0$. 
Let $X\in\R\bP^2$ be an 
intersection of $\overline{\A}$ such that 
$|\overline{\A}_X|=4$ and $|\overline{\scS}_X|=2$. Then the two lines 
of $\overline{\scS}_X$ are adjacent as 
Figure \ref{fig:4cases} $(iii)$. 
In particular, $(iv)$ does not happen. 
\end{Theorem}

\begin{proof}
Suppose that there exists $X\in\R\bP^2$ such that 
$\overline{\A}_X=\{\overline{H}_0, \overline{H}_1, \overline{H}_2, 
\overline{H}_3\}$ with $\overline{\scS}_X=\{\overline{H}_1, \overline{H}_3\}$ 
arranging as $(iv)$ in Figure \ref{fig:4cases}. 

First consider the deconing with respect to $\overline{H}_0$, we have 
$\A=\dec_{\widetilde{H}_0}\widetilde{\A}=\{H_1, \dots, H_n\}$. 
\begin{figure}[htbp]
\begin{picture}(360,170)(0,0)
\thicklines

\color{red}
\put(34,162){$H_1$}
\put(40,10){\line(0,1){150}}

\put(114,162){$H_3$}
\put(120,10){\line(0,1){150}}

\put(130,40){\line(-2,1){110}}

\put(132,50){$H_{i_0}$}
\put(130,55){\line(-6,-1){110}}

\put(130,135){\line(-6,-1){110}}
\normalcolor

\put(74,162){$H_2$}
\put(80,10){\line(0,1){150}}

\put(130,75){\line(-3,1){110}}
\put(130,100){\line(-1,0){110}}

{\small 
\put(44,15){$U_1(B_1)$}
\put(84,15){$U_1(B_2)$}
\put(44,145){$U_2(B_1)$}
\put(84,145){$U_2(B_2)$}
}
{\footnotesize
\put(83,51){$C$}
}

\put(234,162){$H_2'$}
\put(240,10){\line(0,1){150}}

\put(314,162){$H_0'$}
\put(320,10){\line(0,1){150}}

\color{red}
\put(274,162){$H_3'$}
\put(280,10){\line(0,1){150}}

\put(330,65){\line(-3,1){110}}
\put(330,110){\line(-1,0){110}}
\normalcolor

\put(330,25){\line(-3,1){110}}
\put(330,55){\line(-6,-1){110}}
\put(330,135){\line(-6,-1){110}}

{\small 
\put(244,15){$U_1(B_2')$}
\put(284,15){$U_1(B_3')$}
\put(244,145){$U_2(B_2')$}
\put(284,145){$U_2(B_3')$}
}
{\footnotesize
\put(243,42.5){$C'$}
}

\end{picture}
      \caption{Deconings $\dec_{\widetilde{H}_0}\widetilde{\A}$ 
and $\dec_{\widetilde{H}_1}\widetilde{\A}$}
\label{fig:deconings}
\end{figure}
Then $\scS=\{H_1, H_3, \dots\}\subset\A$. 
The lines $H_1, H_2, H_3$ are parallel 
(the left of Figure \ref{fig:deconings}) and 
determines two bands $B_1$ (bounded by $H_1$ and $H_2$) 
and $B_2$ (bounded by $H_2$ and $H_3$). 
Note that 
$e(\scS)=e_1+e_3+\dots\in A_{\bF_2}^1(\A)$. 
By the correspondence in \S\ref{subsec:subarr} (d), we have 
\[
\Psi^{-1}(e(\scS))=[B_1]+[B_2]+\dots, 
\]
in particular, both $[B_1]$ and $[B_2]$ appear. 
(Otherwise, $e_1, e_3$ can not appear.) 
On the other hand, we have the following relation 
\begin{equation}
\widetilde{\nabla}_{{\eta}_0}(\Psi^{-1}(e(\scS))=
\widetilde{\nabla}_{{\eta}_0}([B_1])+
\widetilde{\nabla}_{{\eta}_0}([B_2])+\dots=0. 
\end{equation}
Choose a chamber $C$ such that $C\subset B_2$ and 
$d(U_1(B_2), C)=1$. Let $\Sep(U_1(B_2), C)=\{H_{i_0}\}$. 
The chamber $C$ is adjacent to an unbounded chamber $U_1(B_2)$, 
hence, $C$ is contained in at most two bands $B_2$ and $B_{j_0}$. 
Since $\widetilde{\nabla}_{\widetilde{\eta}}([B_2])=
[C]+\dots\in\bF_2[\RB_{\widetilde{\eta}}(\A)]$, to be $(ii)$ true, 
$[C]$ must be cancelled by another resonant band $B_{j_0}$ which 
appears in $\Psi^{-1}(e(\scS))$. Thus we have 
$\Psi^{-1}(e(\scS))=[B_1]+[B_2]+\dots+[B_{j_0}]+\dots$. 
Let $H_{i_0}$ and $H_{i_0+1}$ be walls of $B_{j_0}$. Then applying 
$\Psi$, we have 
\[
\begin{split}
e(\scS)&=(e_1+e_2)+(e_2+e_3)+\dots+(e_{i_0}+e_{i_0+1})+\dots\\
&=e_1+e_3+\dots+e_{i_0}+\dots. 
\end{split}
\]
Here note that $e_{i_0}$ survives because $B_{j_0}$ is the 
only band which has $H_{i_0}$ as a wall. This implies $H_{i_0}\in\scS$. 
Therefore, if $C\subset B_2$ and 
$d(U_1(B_2), C)=1$, then $\Sep(U_1(B_2), C)\subset\scS$. 
(Left hand side of Figure \ref{fig:deconings}.) 
The same assertion holds for the opposite unbounded chamber 
$U_2(B_2)$. 

Next we consider $\overline{\scS}':=\overline{\A}\setminus
\overline{\scS}$. Since $\widetilde{e}(\widetilde{\scS}')=
\widetilde{\eta}_0+\widetilde{e}(\widetilde{\scS})$, 
$\widetilde{\eta}_0\wedge\widetilde{e}(\widetilde{\scS}')=0$. 
In Figure \ref{fig:4cases} $(iv)$, 
the role of black and red lines exchange. Black lines are 
the member of $\overline{\scS}'$ and red lines are not. 
We take deconing with respect to $\widetilde{H}_1$, we have 
$\dec_{\widetilde{H}_1}\widetilde{\A}=\{H_0', H_2', H_3', \dots, H_n'\}$ 
(Right hand side of Figure \ref{fig:deconings}). 
Then 
$\scS'=\{H_0', H_2', \dots\}\subset\dec_{\widetilde{H}_1}\widetilde{\A}$.  
The lines $H_0', H_2', H_3'$ are parallel and 
determines two bands $B_2'$ (bounded by $H_2'$ and $H_3'$) 
and $B_3'$ (bounded by $H_3'$ and $H_0'$). 
By a similar argument to the previous case (deconing with respect to 
$\widetilde{H}_0$), we can conclude that 
if $C'\subset B_2'$ and 
$d(U_1(B_2'), C')=1$, then $\Sep(U_1(B_2'), C')\subset\scS'$. 
(Right hand side of Figure \ref{fig:deconings}.) 
The same assertion holds for the opposite unbounded chamber 
$U_2(B_1')$. 

The bands $B_2$ and $B_2'$ are identical in the projective 
plane $\R\bP^2$. However, the colors of boundaries of 
unbounded chambers are different. This is a contradiction. 
Thus the case $(iv)$ can not happen. 
\end{proof}

\subsection{Real $4$-nets do not exist}
\label{subsec:4net}

\begin{Theorem}
\label{theo:nonexist4net} 
There does not exist a real arrangement $\overline{\A}$ that 
supports a $4$-net structure.
\end{Theorem}
\begin{proof} 
Suppose $\overline{\A}$ supports a $4$-net structure with partition 
$\overline{\A}=\overline{\A}_1\sqcup\overline{\A}_2\sqcup
\overline{\A}_3\sqcup\overline{\A}_4$. 
There exists a multiple point $X\in\R\bP^2$ 
of $\overline{\A}$ with 
multiplicity $4$ such that $X$ 
is the intersection point of $4$ lines $H_i\in\A_i$. 
Suppose that the lines are ordered like in Figure \ref{4multpointarrfig}. 

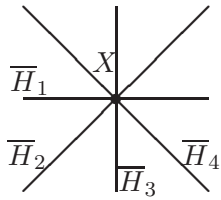
\begin{figure}[htbp]
\begin{picture}(90,70)(0,0)
\thicklines

\put(10,0){\line(1,1){70}}
\put(10,35){\line(1,0){70}}
\put(45,0){\line(0,1){70}}
\put(80,0){\line(-1,1){70}}

\put(45,35){\circle*{4}}
\put(36,45){{\small $X$}}

{\small 
\put(4,10){$\overline{H}_2$}
\put(70,10){$\overline{H}_4$}
\put(5,38){$\overline{H}_1$}
\put(46,0){$\overline{H}_3$}
}
\end{picture}
      \caption{Local structure of a $4$-net.}
\label{4multpointarrfig}
\end{figure}
We can now define $\widetilde{\scS}=\widetilde{\A}_1\sqcup\widetilde{\A}_3$. 
Then as in Example \ref{ex:partition}, we have 
$\widetilde{\eta}_0\wedge\widetilde{e}(\widetilde{\scS})=0$. 
By definition, $\overline{\scS}_X=\{\overline{H}_1, \overline{H}_3\}$ 
consists of two lines and separated by the other two lines 
$\overline{H}_2$ and $\overline{H}_4$. Therefore $(iv)$ in 
Figure \ref{fig:4cases} happens. 
This contradicts the 
Non-separation 
Theorem \ref{thm:nonsep}. 
\end{proof}

\begin{Remark}
\label{rem:pseudo}
The non-existence of real $4$-nets was proved in \cite[Lem. 2.4]{cor-for}. 
Their proof relies on the metric structure of $\bR^2$. So it is not applied to 
oritented matroids. 
Our arguments actually prove that there do not exist 
rank $3$ oriented matroids 
(equivalently, pseudo-line arrangements in $\bR\bP^2$) which have 
$4$-net structures. The details are omitted. 
\end{Remark}

\medskip

\noindent
\textbf{Acknowledgments:} 
The authors thank Alex Suciu for explaining the ideas 
of the preprint \cite{ps-mod}. The authors also thank Mike Falk and Mustafa Hakan Gunturkun 
for telling the reference \cite{cor-for}. During the preparation of this paper, 
the authors were partially supported by 
JSPS Postdoctoral Fellowship For Foreign Researchers, 
the Grant-in-Aid for Scientific Research (C) 25400060, and 
the Grant-in-aid
(No. 23224001 (S)) for Scientific Research, JSPS.

\end{document}